\documentclass{amsart}
\usepackage{amssymb}
\usepackage{amsmath, amscd}
\usepackage{amsthm}
\usepackage{mathrsfs}
\usepackage{perpage}
\usepackage[all,cmtip]{xy}
\usepackage{setspace}
\usepackage{amsbsy}
\usepackage[foot]{amsaddr}
\usepackage{url}
\usepackage[T1]{fontenc}
\usepackage{verbatim}
\usepackage{mathrsfs}
\usepackage{ifthen}
\usepackage{blkarray}
\usepackage{multirow}
\usepackage{enumitem}
\usepackage{chngcntr}
\usepackage{hyperref}
\usepackage{etoolbox}
\usepackage{tikz}
\usetikzlibrary{matrix,arrows,decorations.pathmorphing}
\usepackage[normalem]{ulem}
\usepackage[left=2cm,head=13pt,
top=1.5cm,right=2cm, bottom=1.5cm]{geometry}

\setenumerate[0]{label=(\alph*)}

\numberwithin{equation}{section}

\AtBeginEnvironment{subequations}{\setcounter{equation}{\value{subsection}}}{}{}
\AtEndEnvironment{subequations}{\stepcounter{equation}}{}{}

\AtBeginEnvironment{equation}{\setcounter{equation}{\value{subsection}}}{}{}
\AtEndEnvironment{equation}{\stepcounter{subsection}}{}{}

\newtheorem{theorem}[subsection]{Theorem}
\newtheorem{lemma}[subsection]{Lemma}

\newtheorem{corollary}[subsection]{Corollary}
\newtheorem{definition}[subsection]{Definition}

\theoremstyle{remark}

\newtheorem{rmk}[subsection]{Remark}

\title{Quasi-constant fundamental weights in terms of Levi Weyl groups}

\numberwithin{equation}{section}

\newskip\procskipamount
\newskip\interskipamount
\newskip\refskipamount
\newcommand{\procskip}{\vskip\procskipamount}
\newcommand{\interskip}{\vskip\interskipamount}
\newcommand{\refskip}{\vskip\refskipamount}

\newcommand{\procbreak}{\par
   \ifdim\lastskip<\procskipamount\removelastskip
   \penalty-100
   \procskip\fi
   \noindent\ignorespaces}

\newcommand{\titlebreak}{\par%
\ifdim\lastskip<\interskipamount\removelastskip%
\penalty10000%
\interskip\fi%
\noindent}%

\newcommand{\interbreak}{\par%
\ifdim\lastskip<\interskipamount\removelastskip%
\penalty-100%
\interskip\fi%
\noindent\ignorespaces}%

\newcommand{\refbreak}{\par%
\ifdim\lastskip<\refskipamount\removelastskip%
\penalty-100%
\refskip\fi%
\noindent\ignorespaces}%

\newcounter{listcounter}
\newcounter{deflistcounter}
\newcounter{equivcounter}

\newskip{\itemsepamount}
\newskip{\topsepamount}
\newenvironment{assertionlist}{%
  \begin{list}
    {\upshape (\arabic{listcounter})}
    {\setlength{\leftmargin}{18pt}
     \setlength{\rightmargin}{0pt}
     \setlength{\itemindent}{0pt}
     \setlength{\labelsep}{5pt}
     \setlength{\labelwidth}{13pt}
     \setlength{\listparindent}{\parindent}
     \setlength{\parsep}{0pt}
     \setlength{\itemsep}{\itemsepamount}
     \setlength{\topsep}{\topsepamount}
     \usecounter{listcounter}}}
  {\end{list}}

\newenvironment{definitionlist}{%
  \begin{list}
    {\upshape (\alph{deflistcounter})}
    {\setlength{\leftmargin}{18pt}
     \setlength{\rightmargin}{0pt}
     \setlength{\itemindent}{0pt}
     \setlength{\labelsep}{5pt}
     \setlength{\labelwidth}{13pt}
     \setlength{\listparindent}{\parindent}
     \setlength{\parsep}{0pt}
     \setlength{\itemsep}{\itemsepamount}
     \setlength{\topsep}{\topsepamount}
     \usecounter{deflistcounter}}}
  {\end{list}}

\newenvironment{equivlist}{%
  \begin{list}
    {\upshape (\roman{equivcounter})}
    {\setlength{\leftmargin}{18pt}
     \setlength{\rightmargin}{0pt}
     \setlength{\itemindent}{0pt}
     \setlength{\labelsep}{5pt}
     \setlength{\labelwidth}{13pt}
     \setlength{\listparindent}{\parindent}
     \setlength{\parsep}{0pt}
     \setlength{\itemsep}{\itemsepamount}
     \setlength{\topsep}{\topsepamount}
     \usecounter{equivcounter}}}
  {\end{list}}

% ...

\newcommand{\GG}{\mathbf{G}}

\newcommand{\QQ}{\mathbf{Q}}
\newcommand{\RR}{\mathbf{R}}

\newcommand{\TT}{\mathbf{T}}

\newcommand{\ZZ}{\mathbf{Z}}

\newcommand{\sgn}{\textnormal{sgn}}
\newcommand{\Sgn}{\mathsf{Sgn}}

\newcommand{\even}{\textnormal{even}}

\newcommand{\height}{\textnormal{ht}}

\DeclareMathOperator{\dom}{dom}
\newcommand{\leftexp}[2]{{\vphantom{#2}}^{#1}{#2}}

\DeclareMathOperator{\ad}{ad}

\DeclareMathOperator{\type}{type}

\newcommand{\gal}{{\rm Gal}}

\newcommand{\galk}{\gal(\overline{k}/k)}

\newcommand{\Th}{{\rm Th.}}

\newcommand{\Prop}{{\rm Prop.}}

\newcommand{\loccit}{{\em loc.\ cit. }}

\newcommand{\cf}{{\em cf. }}

\newcommand{\ie}{i.e.,\ }

\AtEndDocument{\bigskip{\footnotesize%
\noindent (Wushi Goldring)
  \textsc{Department of Mathematics, Stockholm University, Stockholm SE-10691, Sweden} \par
 \noindent \texttt{wushijig@gmail.com} \par
 }}

\date{\today}
\author{Wushi Goldring}
\begin{document}

\pagestyle{plain}
\maketitle
\section{Introduction}
\label{sec-intro}
\subsection{Quasi-constancy}
Let $k$ be a field. Let $\GG$ be a connected, reductive $k$-group with maximal $k$-torus $\TT \subset \GG$, associated root datum $\mathsf{RD}(\GG, \TT)=(X^*(\TT), \Phi; X_*(\TT), \Phi^{\vee})$, perfect pairing $\langle, \rangle: X^*(\TT) \times X_*(\TT) \to \ZZ$ and Weyl group $W=W(\Phi)$. Motivated by the root-theoretic properties of the Hodge line bundle on a Hodge-type Shimura variety, Koskivirta and the author introduced in \cite[\S N.5.1]{Goldring-Koskivirta-Strata-Hasse} the purely group-theoretic (or root-data-theoretic) notion of a \emph{quasi-constant} character $\chi \in X^*(\TT)$ or cocharacter $\mu \in X_*(\TT)$. Recall:
\begin{definition} A character $\chi \in X^*(\TT)$ is quasi-constant if, for all $\alpha \in \Phi$ with $\langle \chi, \alpha^{\vee} \rangle \neq 0$ and every $\sigma \in W \rtimes \galk$, one has \begin{equation}
\label{eq-def-quasi-constant}
\frac{\langle \chi, \sigma\alpha^{\vee} \rangle }{\langle \chi, \alpha^{\vee} \rangle}
\in \{-1,0,1\}.
\end{equation}
\label{def-quasi-constant}
\end{definition}
In \cite{Goldring-Koskivirta-quasi-constant}, the authors classified quasi-constant (co)characters and showed that the notion `quasi-constant' naturally unifies those of minuscule and co-minuscule. In particular, if $k$ is algebraically closed and $\GG$ is simple, then a character $\chi \in X^*(\TT)$ is quasi-constant  if and only if it is a multiple of a fundamental weight (relative to some choice of simple roots $\Delta \subset \Phi$) which is either minuscule or co-minuscule \cite[\Th~1.2.1]{Goldring-Koskivirta-quasi-constant}. When $k$ is algebraically closed, the Galois group is trivial and the quasi-constant condition~\ref{def-quasi-constant} depends only on the root system associated to $\mathsf{RD}(\GG, \TT)$. In this note, we henceforth assume $k$ is algebraically closed and use the language of root systems to stress that the additional data in $\mathsf{RD}(\GG, \TT)$ does not play a role here.

In \cite{Goldring-Koskivirta-quasi-constant}, see esp. \S5.2, we argued why it seemed that `quasi-constant' was perhaps a more natural notion than either `minuscule' or `co-minuscule' separately. The goal of this note is to illustrate yet another way in which the quasi-constant condition is natural, by showing that it is equivalent to a property of Weyl groups of maximal Levi subgroups which we now describe.

\subsection{The action of maximal Levi Weyl groups on simple roots}
Let $(V, \Phi, (,))$ be a reduced and irreducible root system with Weyl group $W:=W(\Phi)$, where $V$ is a finite-dimensional $\QQ$-vector space, $(,)$ is a non-degenerate, $\QQ$-valued, $W$-invariant, symmetric bilinear form on $V$ which is positive definite on $V_{\RR}:=V \otimes_{\QQ}\RR$ and $\Phi \subset V \setminus\{0\}$ is the set of roots. For instance, in terms of the root datum $\mathsf{RD}(\GG, \TT)$, assuming $k $ is algebraically closed and the adjoint group $\GG^{\ad}$ of $\GG$ is simple, one may take $V$ to be the $\QQ$-span of $\Phi$ and take $(,)$ to be any $W$-invariant, positive definite symmetric bilinear form (which is unique up to positive scalar).   Let $\Delta \subset \Phi$ be a base of simple roots and choose $\alpha \in \Delta$. Let $W_{\alpha}$ be the Levi Weyl group of the maximal sub-root system generated by $\Delta \setminus \{\alpha \}$ \ie $W_{\alpha}$ is generated by the simple root reflections $s_{\alpha'}$ with $\alpha' \in \Delta \setminus \{\alpha\}$. In the literature, the subgroups $W_{\alpha}$ are often also referred to as maximal parabolic subgroups of $W$. Given $v \in V$, write $\dom(v)$ (resp. $\dom_{\alpha}(v)$) for the unique $\Delta$-dominant (resp. $\Delta \setminus \{\alpha\}$-dominant) conjugate of $v$ under $W$ (resp. $W_{\alpha}$).

It is natural to ask when one has $\dom(\alpha)=\dom_{\alpha}(\alpha)$. That is, by definition there exists $w \in W$ such that $w\alpha=\dom(\alpha)$, but when can the same action be achieved by some $w' \in W_{\alpha}$?

Let $(\eta(\alpha))_{\alpha \in \Delta}$ be the basis of $V$ of fundamental weights, \ie the dual basis of $\Delta^{\vee}:=\{\alpha^{\vee} | \alpha \in \Delta\}$ relative to $(,)$, where $\alpha^{\vee}:=2\alpha/(\alpha, \alpha)$ is the coroot of $\alpha$. Our main result is:
\begin{theorem}
\label{th-levi-dom-quasi-constant}
Let $\alpha \in \Delta$. The following are equivalent:
\begin{enumerate}
\item
\label{item-fund-weight-quasi-constant}
The fundamental weight $\eta(\alpha)$ is quasi-constant.
\item
\label{item-special-cospecial}
The simple root $\alpha$ is either special or co-special (see \ref{sec-special-cospecial}).
\item
\label{item-dom-dom}
One has  $\dom(\alpha)=\dom_{\alpha}(\alpha)$.
\end{enumerate}
\end{theorem}
In view of the previously mentioned classification~\cite[\Th~1.2.1]{Goldring-Koskivirta-quasi-constant},  the equivalence of~\ref{item-fund-weight-quasi-constant} and~\ref{item-special-cospecial} is just a matter of translating between definitions; it was noted in \S2.2.3 of \loccit and is included for convenience. 
The content of the theorem is the equivalence of~\ref{item-fund-weight-quasi-constant},~\ref{item-special-cospecial} with~\ref{item-dom-dom}.
\subsection{Outline}
\label{sec-outline}
\S\ref{sec-notation} introduces notation and recalls pertinent basic results on Weyl group orbits and (co)root multiplicities. Theorem~\ref{th-levi-dom-quasi-constant} is proved in \S\ref{sec-proof}. Examples of Theorem~\ref{th-levi-dom-quasi-constant} are given in \S\ref{sec-examples}. 
\section*{Acknowledgements}
We discovered Theorem~\ref{th-levi-dom-quasi-constant} in the examples of \S\ref{sec-examples} when attempting to understand the work of Jean-Stefan Koskivirta on the "global sections cone" of stacks of $G$-Zips, see \cite{Koskivirta-automforms-GZip}. In turn \loccit is an attempt to better understand our joint conjecture \cite[2.1.6]{Goldring-Koskivirta-global-sections-compositio}, which in particular implies that the global sections cone of a Shimura variety of Hodge type is equal to the corresponding (entirely group-theoretic) $G$-Zip global sections cone.  I'm very grateful to Jean-Stefan for our continued collaboration and discussions related to this note. 

We thank Eric Opdam for discussions related to this note, in particular for introducing us to Lusztig's notion of \textit{special} element in an affine Weyl group \cite[1.11]{Lusztig-unip-p-adic-IMRN-special}. We hope to establish a connection between Lusztig's notion and `quasi-constant' in future work.

We thank the Knut \& Alice Wallenberg Foundation for its support under grants KAW 2018.0356 and Wallenberg Academy Fellow KAW 2019.0256.

We thank the referee for their careful reading and for pointing  us to the references \cite{Landsberg-Manivel-projective-geometry-homogeneous-varieties,wolf-book-constant-curvature}.

\section{Notation and review}
\label{sec-notation}

 Everything we will use about root systems (and much more) is contained in \cite{bourbaki-lie-4-6}.
\subsection{Weyl group orbits}
 \label{sec-weyl-orbits}
Recall that two roots $\alpha ,\beta$ in the reduced and irreducible system $\Phi$ have the same length if and only if they are conjugate under the Weyl group $W$. One says $ \Phi$ (resp. $\Delta$)  is \underline{simply-laced} if $W$ acts transitively on $\Phi$; this corresponds to no two vertices in the Dynkin diagram associated to $\Delta$ being linked by more than one edge. Otherwise one says $\Phi$ is \underline{multi-laced}, there are precisely two Weyl group orbits in $\Phi$ consisting of long and short roots respectively and the associated Dynkin diagram contains a unique pair of adjacent vertices connected by either $2$ edges (types $B_n, C_n$ for $n \geq 2$ and $F_4$) or $3$ edges (type $G_2$) . We adopt the convention that when $\Phi$ is simply-laced, every root is both long and short.

\subsection{Root multiplicities}
\label{sec-root-multiplicities}
Given $\beta \in \Phi$, write $m_{\beta}(\alpha)$ (resp. $m^{\vee}_{\beta}(\alpha)$) for the multiplicity of $\alpha$ in $\beta$ (resp. the multiplicity of $\alpha^{\vee}$ in $\beta^{\vee}$), so that \begin{equation}
\label{eq-root-mult}
\beta
=
\sum_{\alpha \in \Delta}m_{\beta}(\alpha) \alpha
\hspace{.5cm}
\textnormal{ and }
\hspace{.5cm}
\beta^{\vee}
=
\sum_{\alpha \in \Delta}m^{\vee}_{\beta}(\alpha)\alpha^{\vee}.
\end{equation}
Recall that $\Phi \to \Phi^{\vee}$, $\alpha \to \alpha^{\vee}$ is not additive when $\Phi$ is multi-laced, so that in general the multiplicities $m_{\beta}(\alpha)$ and $m^{\vee}_{\beta}(\alpha)$ are different.
We say that $\alpha \in \Delta$ appears in $\beta \in \Phi$ if the multiplicity $m_{\beta}(\alpha) \neq 0$.
\subsection{Highest root and dual of the highest coroot}
\label{sec-highest-root}
Denote by $\Phi^+ \subset \Phi$ the positive roots corresponding to the base $\Delta \subset \Phi$.
Recall that there is a unique root $\alpha^h$, called the highest root, which is characterized by $m_{\alpha^h}(\alpha)\geq m_{\beta}(\alpha) $ for all positive roots $\beta \in \Phi^+$ and all $\alpha \in \Delta$.
 When $\beta=\alpha^h$ is the highest root, we abbreviate $m(\alpha):=m_{\alpha^h}(\alpha).$

 Write $\leftexp{h}{\alpha}^{\vee}$ for the highest coroot, \ie the the highest root in $\Phi^{\vee}$; the dual $\alpha^{h_2}:=(\leftexp{h}{\alpha}^{\vee})^{\vee}$ is a root which is $\Delta$-dominant. In view of \S\ref{sec-weyl-orbits}, one has $\alpha^{h_2}=\alpha^h$ if and only if $\Phi$ is simply-laced; if $\Phi$ is multi-laced then $\alpha^{h_2}$ is the highest short root. When $\beta^{\vee}=\leftexp{h}{\alpha}^{\vee}$ is the highest coroot, we abbreviate $m^{\vee}(\alpha):=m^{\vee}_{\alpha^{h_2}}(\alpha)$ the multiplicity of $\alpha^{\vee}$ in the highest coroot $\leftexp{h}{\alpha}^{\vee}=(\alpha^{h_2})^{\vee}$.
 
 If $\Phi$ is simply-laced, then the highest root $\alpha^h$ is the unique $\Delta$-dominant root and $\dom(\alpha)=\alpha^h$ for all $\alpha \in \Phi$.
Assume $\Delta$ is multi-laced. Then $\alpha^h$ and $\alpha^{h_2}$ are the two and only two $\Delta$-dominant roots.
For $\alpha \in \Phi$, one has $\dom(\alpha)=\alpha^h$ if and only if $\alpha$ is long and $\dom(\alpha)=\alpha^{h_2}$ if and only if $\alpha$ is short.

\subsection{Height}
\label{sec-height}
Recall that the \underline{height} of a positive root $\beta \in \Phi^+$ is $\height(\beta):=\sum_{\alpha \in \Delta}m(\alpha)$, the sum of the root multiplicities.
\subsection{Special and co-special simple roots}
\label{sec-special-cospecial}
In \cite[1.2.5]{Deligne-Shimura-varieties}, Deligne calls a simple root $\alpha \in \Delta$ \underline{special} if its multiplicity $m(\alpha)=1$; by analogy Koskivirta and the author termed $\alpha$ \underline{co-special} if the dual multiplicity $m^{\vee}(\alpha)=1$, \cite[\S2.1.6]{Goldring-Koskivirta-quasi-constant}. By definition $\alpha$ is special (resp. co-special) if and only if $\alpha^{\vee}$ is co-special (resp. special).

\section{Proof of the theorem}
\label{sec-proof}
\begin{lemma}
\label{lem-root-pairings}
Assume $(V, \Phi, (,))$ is a reduced and irreducible root system with base $\Delta \subset \Phi$.  Consider a simple root $\alpha \in \Delta$ and a long, positive root $\beta \in \Phi^+$. Assume that: \begin{enumerate}
\item
\label{item-lem-unique-not-orthogonal-root}
One has $(\beta, \alpha') \leq 0$ for all $\alpha' \in \Delta \setminus \{\alpha\}$;
\item
\label{item-lem-mult-at-most-1}
$\alpha$ appears at most once in $\beta$: The multiplicity $m_{\beta}(\alpha) \leq 1$.
\end{enumerate}
Then $\beta=\alpha$.
\end{lemma}
\begin{rmk}
\label{rmk-landsberg-manivel-wolf}
We thank the referee for pointing out that Lemma~\ref{lem-root-pairings} is closely related to an unpublished result of Kostant on modules of Lie algebras (\cf \cite[\Prop~2.4]{Landsberg-Manivel-projective-geometry-homogeneous-varieties} and \cite[\Th~8.13.3]{wolf-book-constant-curvature}; the methods of proof there are quite different from ours). The lemma is also closely related to \cite[Lemma 4.4]{Landsberg-Manivel-projective-geometry-homogeneous-varieties}; in fact the latter can be used to give a variant of our proof that \ref{item-special-cospecial}$\Rightarrow$\ref{item-dom-dom} in Theorem~\ref{th-levi-dom-quasi-constant}. 
\end{rmk}
\begin{rmk}
\label{rmk-unique-root-not-orthogonal}
Since $(\beta, \beta)>0$, ~\ref{lem-root-pairings}\ref{item-lem-unique-not-orthogonal-root} implies that $(\beta, \alpha)>0$, so that
$\alpha$ is the unique simple root which has strictly positive pairing with $\beta$.
\end{rmk}
\begin{proof}[Proof of Lemma~\ref{lem-root-pairings}:]
Put $\beta':=\beta-\alpha$. We show $\beta'=0$. Since $(\alpha', \alpha'') \leq 0$ for distinct simple roots $\alpha', \alpha'' \in \Delta$, assumptions~\ref{item-lem-unique-not-orthogonal-root}-\ref{item-lem-mult-at-most-1}
together imply that
$(\beta', \alpha) \leq 0$
and that
$(\beta, \beta)
\leq
(\beta, \alpha)$. Since $(\beta, \beta)>0$, ~\ref{item-lem-unique-not-orthogonal-root} also gives $m_{\beta}(\alpha)=1$. Since $(\beta, \alpha)=(\alpha, \alpha)+(\beta', \alpha)$, one has $(\beta, \beta) \leq (\beta, \alpha) \leq (\alpha, \alpha)$. Since $\beta$ is long,
$\alpha$ must be long too.
Hence $$(\beta, \beta)=(\beta,\alpha)=(\alpha, \alpha)$$ and $(\beta', \alpha)=0$. Since $(\beta, \beta)=(\alpha, \alpha)+(\beta',\beta')+2(\beta', \alpha)$, we conclude that $(\beta',\beta')=0$. So $\beta'=0$.

\end{proof}

\begin{lemma}
\label{lem-special-weyl-conj-long-root}
Assume $\alpha \in \Delta$ is special and $\beta \in \Phi$ is long. If $\alpha$ appears in $\beta$ (\ie $m_{\beta}(\alpha) \neq 0$), then $\beta$ is $W_{\alpha}$-conjugate to $\alpha$.
 \end{lemma}
\begin{proof}
Upon replacing $\beta$ by its negative if necessary, we henceforth assume $\beta$ is positive. We argue by induction on the height $\height(\beta)$ (\S\ref{sec-height}). If $\height(\beta)=1$, then $\beta=\alpha$ since $\alpha$ appears in $\beta$.

Assume $\height(\beta)>1$. Since $\alpha$ is special,~\ref{lem-root-pairings}\ref{item-lem-mult-at-most-1} holds. Since $\beta \neq \alpha$, Lemma~\ref{lem-root-pairings} implies that~\ref{lem-root-pairings}\ref{item-lem-unique-not-orthogonal-root} fails: There exists $\alpha' \in \Delta \setminus \{\alpha\}$ with $(\beta, \alpha')>0$.

We check that the root $s_{\alpha'}(\beta)=\beta-\langle \beta, (\alpha')^{\vee}\rangle \alpha'$ satisfies the induction hypothesis: Since the Weyl group preserves length, $s_{\alpha'}(\beta)$ is long.
One has $\alpha' \neq \beta$ as $\height(\beta)>1$.
Since $\alpha'$ is simple and $\beta$ positive, the reflection $s_{\alpha'}(\beta)$ is again positive.  Since $(\beta, \alpha')>0$, also $\langle \beta, (\alpha')^{\vee} \rangle >0$,
so that $\height(s_{\alpha'}(\beta))<\height(\beta)$. Finally, $\alpha' \neq \alpha$ implies that the multiplicity of $\alpha$ is unchanged: $m_{s_{\alpha'}(\beta)}(\alpha)=m_{\beta}(\alpha)=1$. By induction $s_{\alpha'}(\beta)$ is $W_{\alpha}$-conjugate to $\alpha$; since $s_{\alpha' } \in W_{\alpha}$, we conclude that also $\beta$ is $W_{\alpha}$-conjugate to $\alpha$.
\end{proof}
Since the highest root is long \cite[VI.1.8, \Prop~25(iii)]{bourbaki-lie-4-6}, in particular:
\begin{corollary}
\label{cor-special-conjugate-highest-root}
Assume $\alpha \in \Delta$ is special. Then $\alpha$ is $W_{\alpha}$-conjugate to the highest root $\alpha^h$.
\end{corollary}

\begin{lemma}
\label{lem-levi-weyl-preserves-root-mult}
Let $\alpha \in \Delta$. Assume two roots $\beta, \gamma \in \Phi$ are $W_{\alpha}$-conjugate. Then $\alpha$ appears with the same multiplicity in $\beta$ and $\gamma$, \ie $m_{\beta}(\alpha)=m_{\gamma}(\alpha)$.
\end{lemma}
\begin{proof}
If $\alpha' \in \Delta \setminus \{\alpha\}$, then applying $s_{\alpha'}$ to $\beta$ only alters the multiplicity of $\alpha'$ in $\beta$; in particular the multiplicity of $\alpha$ is unchanged. The result follows since $W_{\alpha}$ is generated by the $s_{\alpha'}$ with $\alpha' \in \Delta \setminus \{\alpha\}$.
\end{proof}
\begin{proof}[Proof of Theorem~\ref{th-levi-dom-quasi-constant}:]
As mentioned right after the statement of~\ref{th-levi-dom-quasi-constant}, it suffices to prove \ref{item-special-cospecial}$\Leftrightarrow$\ref{item-dom-dom}.
Let $\alpha \in \Delta$.
Assume first~\ref{item-dom-dom}, \ie that $\dom(\alpha)=\dom_{\alpha}(\alpha)$. By Lemma~\ref{lem-levi-weyl-preserves-root-mult}, $m_{\dom(\alpha)}(\alpha)=m_{\alpha}(\alpha)=1$.
If $\alpha$ is long, then $\dom(\alpha)=\alpha^h$ is the highest root, so $\alpha$ is special. If $\alpha$ is short, then $\alpha^{\vee}$ is long. Applying the previous argument in the dual root system $\Phi^{\vee}$ gives that $\alpha^{\vee}$ is special in $\Phi^{\vee}$ \ie that $\alpha$ is co-special (\S\ref{sec-special-cospecial}).

Next, assume $\alpha$ is special. Then $\dom(\alpha)=\dom_{\alpha}(\alpha)$ by Corollary~\ref{cor-special-conjugate-highest-root}. Finally, suppose $\alpha$ is co-special.
Then $\alpha^{\vee}$ is special in $\Phi^{\vee}$, so by the previous case there exists $w \in W_{\alpha}$ such that $w\alpha^{\vee}=\leftexp{h}{\alpha}^{\vee}$.
But then $w\alpha=\alpha^{h_2}=\dom(\alpha)$, so $\dom_{\alpha}(\alpha)=\dom(\alpha)$.
\end{proof}

\section{Classical examples}
\label{sec-examples}
For root systems of classical type and type $G_2$, we illustrate  Theorem~\ref{th-levi-dom-quasi-constant} in terms of the classical explicit descriptions of such root systems given in the {\em planches} of \cite{bourbaki-lie-4-6}. For the larger exceptional cases, as is so often the case, it seems necessary (or at least much easier) to revert to the general theory.
\subsection{Notation}
\label{sec-notation-examples}
Let $e_i$ be the $i$th standard basis vector of $\QQ^n$. Let $\QQ^n_0:=\{(a_1, \ldots a_n) \in \QQ^n \ | \ \sum_{i=1}^n a_i=0\}$, the hyperplane with vanishing sum of coordinates. Given a finite set $X$, let $S_X$ denote its  symmetric group, $\Sgn_X \cong (\ZZ/2)^{|X|}$ its group of sign changes and $\Sgn_X^{\even} \cong (\ZZ/2)^{|X|-1}$ the subgroup  of even sign changes on $X$; when $X=\{1,2 \ldots, n\}$ write $S_n$, $\Sgn_n$ and $\Sgn_n^{\even}$ respectively.
\subsection{Type $A_{n-1}$, $n \geq 2$} Let $V=\QQ^n_0$, let $\Phi_A=\{e_i-e_j\ | \ i \neq j\}$ and $\Delta_A=\{e_1-e_2, \ldots ,e_{n-1}-e_n\}$.
Then $(V, \Phi_A)$ is simply-laced of type $A_{n-1}$ and $\Delta_A$ is a base for which $\alpha^h=e_1-e_n$.
For all $1 \leq i \leq n-1$, the root $e_i-e_{i+1}$ is special (and co-special),   $W_{e_i-e_{i+1}}=S_i \times S_{\{i+1, \ldots, n\}}$ and $s_{e_1-e_i}s_{e_{i+1}-e_n}=(1 \ i)(i+1 \ n) \in W_{e_i-e{i+1}}$ maps $e_i-e_{i+1}$ to $\alpha^h$.

\subsection{Type $D_{n}$, $n \geq 4$} Let $V=\QQ^n$, let $\Phi_D:=\Phi_A \cup \{\pm(e_i +e_j) | 1 \leq i < j \leq n\}$ and $\Delta_D:=\Delta_A \cup \{e_{n-1}+e_n\}$. Then $(V, \Phi_D)$ is simply-laced of type $D_n$ and $\Delta_D$ is a base for which $\alpha^h=e_1+e_2$.
\subsubsection{Special roots}
There are precisely three special roots: $e_1-e_2$, $e_{n-1}-e_n$ and $e_{n-1}+e_n$; these correspond to the extremities of the Dynkin diagram.
\begin{enumerate}
\item $W_{e_1-e_2}=S_{\{2, \ldots , n \} } \ltimes \Sgn_{\{ 2, \ldots ,n\}}^{\even}$. The (even) sign change $\sgn_{\{2,3\}} \in W_{e_1-e_2}$ maps $e_1-e_2$ to $\alpha^h$.
\item $W_{e_{n-1}-e_n}= \langle S_{n-1}, s_{e_{n-1}+e_n} \rangle$ and
$ (2 \ n-1)  \circ s_{e_{n-1}+e_n} \circ (1 \ n-1) \in W_{e_{n-1}-e_n}$ maps $e_{n-1}-e_n$ to $\alpha^h$.
\item
$W_{e_{n-1}+e_n}=S_n$ and $ (1 \ n-1)(2 \ n) \in W_{e_{n-1}+e_n}$ maps $e_{n-1}+e_n$ to $\alpha^h$.
\end{enumerate}
\subsubsection{Internal roots}
For $1 <i<n-1$, the root $e_i-e_{i+1}$ and its co-root both have multiplicity $2$ (so neither special nor co-special); one has $W_{e_i-e_{i+1}} \cong S_i \times (S_{\{i+1, \ldots, n\}} \ltimes \Sgn^{\even}_{\{i+1, \ldots , n\}})$,
corresponding to $\type(\Delta \setminus \{e_i-e_{i+1}\}) \cong A_{i-1} \times D_{n-i}$.
In particular the action of $W_{e_i-e_{i+1}}$ on $\{1,\ldots , n\}$ leaves stable the decomposition $\{1,\ldots ,n\}=\{1, \ldots, i\} \coprod \{i+1, \ldots, n\}$,
so no element of $W_{e_i-e_{i+1}}$ maps $e_i-e_{i+1}$ to $\alpha^h$.

\subsection{Type $B_{n}$, $n \geq 2$} Let $V=\QQ_n$, let $\Phi_B=\Phi_D \cup \{\pm e_i \ | \ 1 \leq i \leq n\}$ and $\Delta_B=\Delta_A \cup \{e_n\}$. Then $(V, \Phi_B)$ is multi-laced of type $B_n$ and $\Delta_B$ is a base for which $\alpha^h=e_1+e_2$ and $\alpha^{h_2}=e_1$.
\subsubsection{Special and co-special roots}
The root $e_1-e_2$ is the unique special root; $W_{e_1-e_2}= S_{\{2, \ldots, n\}} \ltimes \Sgn_{\{2, \ldots, n\}}$ and  $s_{e_2}=\sgn_{\{2\}} \in W_{e_1-e_2}$  maps $e_1-e_2$ to $\alpha^h$.
The unique co-special root is $e_n$, $W_{e_n} = S_n$ and $s_{e_1-e_n}=(1 \ n) \in W_{e_n}$ maps $\alpha$ to $\alpha^{h_2}$.
\subsubsection{Internal roots}
Assume $n \geq 3$ and $1 < i<n$. Then $e_i-e_{i+1}$ is long; it and its coroot again both have multiplicity $2$ and the argument for types $D_n$ shows that $e_i-e_{i+1}$ is not in the $W_{e_i-e_{i+1}}$-orbit of $\alpha^h$.

\subsection{Type $C_{n}$, $n \geq 3$} Let $V=\QQ^n$, let $\Phi_C:=\Phi_A \cup \{\pm 2e_i \ |\ 1 \leq i \leq n\}$ and $\Delta_C:=\Delta_A \cup \{2e_n\}$. Then $(V, \Phi_C)$ is multi-laced of type $C_n$ and $\Delta_C$ is a base for which  $\alpha^h=2e_1$ and $\alpha^{h_2}=e_1+e_2$. Since $(V, \Phi_C)$ is dual to $(V, \Phi_B)$, the unique special (resp. co-special) root $2e_n$ (resp. $e_1-e_2$) in type $C_n$ is the coroot of the unique co-special (resp. special) root in type $B_n$; the element of $W_{\alpha}$ used to map $\alpha$ to $\dom(\alpha)$ in type $B_n$ also works to map $\alpha^{\vee}$ to $\dom \alpha^{\vee}$ and the same reason explains why $\dom_{\alpha^{\vee}}\alpha^{\vee} \neq \dom \alpha^{\vee}$ for the remaining "internal" simple roots which are neither special nor co-special.
\subsection{Rank $2$} Assume $(V, \Phi)$ is reduced and irreducible of rank $\dim V=2$. Recall this means $(V, \Phi)$ is of one of three types $A_2, B_2 \cong C_2, G_2$. Choose a base $\Delta \subset \Phi$ and write $\Delta=\{\alpha, \beta\}$. Then $W_{\alpha}=\{1, s_{\beta}\}$, so that the condition $\dom_{\alpha}=\dom_{\alpha}\alpha$ of Theorem~\ref{th-levi-dom-quasi-constant}\ref{item-dom-dom} will hold if and only if 
\begin{equation}
\label{eq-rank-2}    
s_{\beta}(\alpha)= 
\left\{ 
\begin{array}{lll}
\alpha^h     &  \textnormal{ if } & \alpha \textnormal{ is long} \\
\alpha^{h_2}    & 
\textnormal{ if } & \alpha \textnormal{ is short.}
\end{array} \right. 
\end{equation} 

In types $A_2, B_2 \cong C_2$, every simple root is either special, co-special or both, so that~\eqref{eq-rank-2} always holds. In type $G_2$, both roots are neither special nor co-special, so that Theorem~\ref{th-levi-dom-quasi-constant} says that~\eqref{eq-rank-2} is never satisfied. This is verified explicitly below.
\subsection{Type $G_2$} Let $V=\QQ_3^0$, $\Phi_G=\Phi_A \cup \{\pm(2e_1-e_2-e_3),\pm(2e_2-e_1-e_3),\pm(2e_3-e_1-e_2)\}$ and $\Delta_G=\{\alpha, \beta\}$ with $\alpha=e_1-e_2$, $\beta=-2e_1+e_2+e_3$. Then $(V,\Phi_G)$ is of type $G_2$, the subsystem $(V,\Phi_A)$ of short roots is of type $A_2$ and $\Delta_G$ is a base of $(V,\Phi_G)$ for which $\alpha^h=2e_3-e_1-e_2=3\alpha+2\beta$ and $\alpha^{h_2}=e_3-e_2=2\alpha+\beta$. Then Theorem~\ref{th-levi-dom-quasi-constant} amounts to:
\begin{subequations}
\begin{align}
s_{\alpha}(\beta)=-2e_2+e_1+e_3=3\alpha+\beta \neq \alpha^h
\\
 s_{\beta}(\alpha)=\alpha+\beta=e_3-e_1 \neq e_3-e_2.   
\end{align}
\end{subequations}

\bibliographystyle{plain}
\bibliography{biblio_overleaf}

\end{document}